\documentclass[11pt, reqno]{amsart}

\usepackage[margin=2.75cm]{geometry}

\usepackage{amsmath, amssymb, amsthm}

\usepackage{xcolor}

\usepackage[colorlinks=true, urlcolor=blue, linkcolor=blue, citecolor=blue, pdfstartview=FitH]{hyperref}

\usepackage{esint}

\newtheorem{theorem}{Theorem}
\newtheorem{lemma}{Lemma}
\newtheorem{proposition}{Proposition}
\newtheorem{definition}{Definition}
\newtheorem{corollary}{Corollary}

\newtheorem{remark}{Remark}

\newtheorem*{ack}{Acknowledgments}

\begin{document}

\title[Volume bound involving Ricci and scalar curvature]{An improved volume bound under Ricci and scalar curvature lower bounds}
\author[K.-K. Kwong]{Kwok-Kun Kwong}
\address{School of Mathematics and Physics\\
University of Wollongong\\
Northfields Ave, Wollongong, NSW $2500$\\
Australia}
\email{kwongk@uow.edu.au}
\date{}
\subjclass[2020]{Primary 53C20; Secondary 53C21, 53C24.}

\keywords{volume comparison, Ricci curvature, scalar curvature, geodesic flow, shuffling lemma, Branson Q-curvature}

\maketitle
\begin{abstract}
We study volume comparison for closed Riemannian manifolds satisfying a positive Ricci curvature lower bound together with an improved scalar curvature lower bound.
We prove that if a closed Riemannian manifold $(N^n, g)$ satisfies $\operatorname{Ric}_g\ge (n-1)g$ and the scalar curvature $R_g\ge n(n-1)(1+\varepsilon)$, then its volume satisfies
$$\lvert N\rvert_g \le \frac{1}{\sqrt{1+n\varepsilon}}\lvert\mathbb S^n\rvert. $$
In fact, assuming only $\mathrm{Ric}_g\ge(n-1)g$, we can prove that
$$\frac{|N|_g}{\left|\mathbb{S}^n\right|} \le \frac{1}{|N|_g} \int_N\left(\frac{R_g}{n-1}-(n-1)\right)^{-\frac{1}{2}} d \mathrm{vol}_g. $$
The equality holds if and only if $N$ is isometric to the unit sphere.

The proof combines a coefficient-adapted Jacobian comparison with a new integral shuffling comparison for scalar Jacobi solutions, inspired by Brown and Freedman \cite{BrownFreedman2022}. This yields an estimate that retains the full Ricci spectrum. The resulting volume bound agrees to first order in $\varepsilon$ with the factor appearing in Bray's conjecture. A further consequence of the argument is an averaged volume comparison for metric balls involving the scalar curvature. We also obtain a volume comparison theorem under a weighted integral lower bound on the Branson $Q$-curvature.
\end{abstract}

\section{Introduction}

Let $(\mathbb S^n, g_{\mathbb S^n})$ denote the unit round sphere. Bishop's volume comparison theorem states that if a closed $n$-dimensional Riemannian manifold\footnote{All manifolds are assumed to be connected in this paper.}
$(N, g)$ satisfies $\operatorname{Ric}_g\ge (n-1)g$, then the volume $|N|_g\le |\mathbb S^n|$. Taking the trace of the Ricci lower bound gives $R_g\ge n(n-1)$. It is therefore natural to ask whether a strictly stronger scalar-curvature lower bound leads to a corresponding improvement of Bishop's volume bound.

In a remarkable thesis that established the Riemannian Penrose inequality and developed volume comparison theorems involving scalar curvature, Bray proved a result of this type in dimension three and proposed a higher-dimensional generalization \cite[Conjecture~4]{Bray}. This conjecture was very recently proved by Jiang, Li, and Wang~\cite{JiangLiWang}. In the normalization used in this paper, the resulting theorem can be stated as follows.

\begin{theorem}[Bray's conjecture \cite{Bray, JiangLiWang}]\label{thm:Bray conjecture}
For each $n\ge 3$, there exists a constant $\Lambda_n>1$, depending only on $n$, such that every closed $n$-dimensional Riemannian manifold $(N, g)$ satisfying
$$
\operatorname{Ric}_g\ge (n-1)g,
\qquad
R_g\ge n(n-1)\Lambda_n,
$$
also satisfies
\begin{equation}\label{eq:bray-scaled}
\frac{|N|_g}{|\mathbb S^n|} \le \Lambda_n^{-\frac n2}.
\end{equation}
\end{theorem}

Equivalently, for each $n \ge 3$, there exists a positive $\epsilon_0(n)<1$ such that if $(N^n, g)$ is any complete smooth Riemannian manifold satisfying
$$
\operatorname{Ric}_g \ge \epsilon_0(n) (n-1)g,
\qquad
R_g \ge n(n-1),
$$
then $|N|_g\le |\mathbb S^n|$.
\medskip
\begin{remark}
\noindent
The first version of the present paper was posted on arXiv on 19 August 2026. The first version of Jiang, Li, and Wang~\cite{JiangLiWang}, which proves Bray's conjecture, was posted on 20 August 2026. The two works were developed independently and use different methods.
\end{remark}
The proof of Jiang, Li, and Wang uses Perelman's entropy, spherical rearrangement, and sharp functional inequalities on the sphere. Shortly after the first version of \cite{JiangLiWang} appeared, Fan and Yu~\cite{FanYu} and Jiang, Li, and Wang~\cite{JiangLiWang} independently established rigidity in the equality case, proving that equality holds if and only if $(N, g)$ is isometric to a round sphere. Jiang, Li, and Wang also obtained an application involving the Branson $Q$-curvature.

Despite the recent resolution of the conjecture, the three-dimensional case remains substantially better understood at a quantitative level. Bray's football theorem \cite[Theorem 19]{Bray} gives a sharp function $\alpha\colon(0, 1]\to[1, \infty)$ such that if $\operatorname{Ric}_g\ge 2\eta g$ and $R_g\ge 6$, then $|N|_g\le \alpha(\eta)|\mathbb S^3|$ for every closed three-manifold $(N^3, g)$. If $\eta_0=\inf\{\eta\in(0, 1]:\alpha(\eta)=1\}$, then $\alpha(\eta)=1$ for $\eta\ge \eta_0$.

According to Bray, numerical analysis suggests that $\eta_0$ lies between $0.134$ and $0.135$. Gursky and Viaclovsky proved rigorously that $\eta_0\le \frac12$ \cite[Section~4.1]{GurskyViaclovsky}. In particular, if $\operatorname{Ric}_g\ge g$ and $R_g\ge 6$, then $|N|_g\le |\mathbb S^3|$.
Brendle treated the equality case and proved that equality under these assumptions forces $(N, g)$ to be isometric to the unit round sphere \cite[Theorem~5.6]{BrendleRigidity}. The argument uses the Gauss equation and the Gauss--Bonnet theorem on two-dimensional isoperimetric surfaces. This dimension-dependent step is one of the main obstacles to carrying the method over directly to dimensions $n\ge 4$.

Some additional assumptions are necessary if one retains only a scalar curvature lower bound. A deformation result of Corvino, Eichmair, and Miao shows that scalar curvature alone does not generally control volume \cite{CorvinoEichmairMiao}. This motivates volume comparison results near special background metrics, such as $V$-static or Einstein metrics.

Before the resolution in \cite{JiangLiWang}, several important special cases and related comparison results were known.
Yuan \cite{Yuan} proved a volume comparison result near a strictly stable Einstein metric. More precisely, suppose that $\operatorname{Ric}_{\bar g}=(n-1)\lambda\bar g$, $\lambda\ne 0$, and that $\bar g$ is strictly stable. Then there is a constant $\delta>0$ such that every metric $g$ satisfying
$$
R_g\ge n(n-1)\lambda,
\qquad
\|g-\bar g\|_{C^2(N, \bar g)}<\delta,
$$
obeys
$$
|N|_g\le |N|_{\bar g}\quad\text{if }\lambda>0,
\qquad
|N|_g\ge |N|_{\bar g}\quad\text{if }\lambda<0.
$$
Equality holds only when $g$ is isometric to $\bar g$ \cite[Theorem~B]{Yuan}. Applying this theorem to the unit round sphere gives
$$
R_g\ge n(n-1)
\quad\Longrightarrow\quad
|\mathbb S^n|_g\le |\mathbb S^n|
$$
whenever $g$ is sufficiently $C^2$-close to the round metric \cite[Corollary~B]{Yuan}. No separate Ricci curvature assumption is needed in this local result, since $C^2$-closeness already implies the Ricci tensor to be close to that of the round sphere. This result concerns metrics in a $C^2$-neighborhood of the round metric.

Zhang \cite{Zhang} obtained two complementary higher-dimensional results. The first is nonperturbative but assumes additional symmetry. He proved that for an axisymmetric metric $g$, for any $n\ge 3$, there exists $\eta_n\in(0, 1)$ such that if $\operatorname{Ric}_g\ge \eta_n(n-1)g$ and $R_g\ge n(n-1)$, then $|N|_g\le |\mathbb S^n|$ \cite[Theorem~3]{Zhang}. Unlike Yuan's result, no closeness to the round metric and no Ricci upper bound are assumed. The symmetry reduces the two curvature inequalities to differential inequalities for the single warping function $f$, from which the volume estimate is obtained directly.

Zhang \cite[Theorem~4]{Zhang} also proved that, for every $C>0$, there exists $\delta=\delta(n, C)\in(0, 1)$ such that any closed $n$-manifold satisfying $(1-\delta)(n-1)g\le \operatorname{Ric}_g\le Cg$ and $R_g\ge n(n-1)$
also satisfies $|N|_g\le |\mathbb S^n|$. The proof relies on the upper Ricci bound in its compactness argument, and the constant $\delta$ is allowed to depend on $C$.

Kwong \cite[Theorem~2.5]{Kwong2025} also obtained an explicit quantitative upper bound for the total volume under an additional Ricci curvature upper bound. More precisely, if $\kappa>0$ and
$0\le \operatorname{Ric}_g-(n-1)g \le \kappa g$,
then $|N|_g$ is bounded above by $|\mathbb S^n|$ minus a nonnegative correction term involving $\kappa$ and the average scalar-curvature excess
$\fint_N\left(R_g-n(n-1)\right)\, d\mathrm{vol}_g$.

The estimates obtained in this paper give explicit quantitative improvements of Bishop's volume bound for every $\varepsilon\ge 0$, without additional symmetry, closeness, or upper curvature assumptions. More importantly, the estimate retains the full Ricci spectrum and is valid under the sole assumption of positive Ricci curvature.

\begin{theorem}[Theorem~\ref{thm:determinant}]
\label{thm:intro-determinant}
Let $(N^n, g)$ be closed with $\operatorname{Ric}_g>0$. Then
\begin{equation}\label{eq:intro-determinant-bound}
\frac{|N|_g}{|\mathbb S^n|}
\le\frac{1}{|N|_g} \int_N \det\nolimits_g \left(\frac{\operatorname{Ric}}{n-1}\right)^{-\frac12} \, d\mathrm{vol}_g.
\end{equation}
The equality holds if and only if $(N, g)$ is isometric to a round sphere of positive constant curvature.
\end{theorem}

Under the stronger assumption $\operatorname{Ric}_g\ge (n-1)g$, the determinant estimate yields an integral volume bound involving only the scalar curvature. Indeed, the Ricci curvature lower bound and the elementary inequality
$\displaystyle \prod_{i=1}^n(1+a_i)\ge 1+\sum_{i=1}^n a_i$, $\displaystyle a_i\ge 0$,
give
$$
\det\nolimits_g\left(\frac{\operatorname{Ric}_g}{n-1}\right)
\ge
\frac{R_g}{n-1}-(n-1).
$$
Together with \eqref{eq:intro-determinant-bound}, this gives the following result.

\begin{theorem}[Theorem~\ref{thm:main}]
\label{thm:intro-main2}
Let $(N^n, g)$ be a closed Riemannian manifold, where $n\ge 3$, and suppose that $\operatorname{Ric}_g\ge (n-1)g$. Then
$$
\frac{|N|_g}{|\mathbb S^n|}
\le
\frac{1}{|N|_g}
\int_N
\left(\frac{R_g}{n-1}-(n-1)\right)^{-\frac12}
\, d\mathrm{vol}_g.
$$
The equality holds if and only if $N$ is isometric to the unit sphere.
\end{theorem}
\begin{remark}
Under the assumption of Theorem~\ref{thm:intro-main2}, define the nonnegative tensor
$E=\frac{\operatorname{Ric}}{n-1}-g$.
If $\sigma_j(E)$ denotes the $j$-th elementary symmetric function of the eigenvalues of $E$ with respect to $g$, with $\sigma_0(E)=1$, then \eqref{eq:intro-determinant-bound} in fact gives
$$
\frac{|N|_g}{|\mathbb S^n|}
\le
\frac{1}{|N|_g}
\int_N
\left(\sum_{j=0}^n\sigma_j(E)\right)^{-\frac12}
\, d\mathrm{vol}_g.
$$
Theorem~\ref{thm:intro-main2} follows by replacing the upper limit $n$ by $1$ in the sum.
\end{remark}
A pointwise scalar-curvature lower bound gives the following explicit consequence.

\begin{theorem}[Corollary~\ref{cor:uniform-scalar-bound}]
\label{thm:intro-main}
Let $(N^n, g)$ be a closed Riemannian manifold, where $n\ge 3$, and suppose that
$$
\operatorname{Ric}_g\ge (n-1)g,
\qquad
R_g\ge n(n-1)(1+\varepsilon)
$$
for some $\varepsilon\ge 0$. Then
\begin{equation}\label{eq:intro-main}
\frac{|N|_g}{|\mathbb S^n|}
\le
(1+n\varepsilon)^{-\frac12}.
\end{equation}
The equality holds if and only if $N$ is isometric to the unit sphere and $\varepsilon=0$.
\end{theorem}
For the connection with Bray's conjecture, we should compare $\Lambda_n$ with $1+\varepsilon$. In contrast to Theorem~\ref{thm:Bray conjecture}, the estimate holds for any arbitrary $\varepsilon\ge 0$.
For $\varepsilon>0$, the right-hand side of \eqref{eq:intro-main} is strictly smaller than one. Thus Theorem~\ref{thm:intro-main} gives a strict improvement over Bishop's volume bound in every dimension. On the other hand, by Bernoulli's inequality, for $\varepsilon>0$,
$$
(1+n\varepsilon)^{-\frac12} > (1+\varepsilon)^{-\frac n2},
$$
so our estimate does not reach the volume factor \eqref{eq:bray-scaled} appearing in Bray's conjecture. Nevertheless, the difference between the two factors only first occurs at order $\varepsilon^2$.

The argument leading to Theorem~\ref{thm:intro-determinant} also yields the following comparison for the average volume of metric balls. It incorporates the scalar-curvature excess and improves the estimate obtained by averaging the Bishop--Gromov comparison over the centre of the ball.

\begin{theorem}[Theorem~\ref{thm:scalar-excess-metric-balls}]
\label{thm:intro-average-balls}
Let $(N^n, g)$ be a closed Riemannian manifold satisfying $\operatorname{Ric}_g\ge (n-1)g$, and define the average volume of metric balls by
$\overline V(r)=\fint_N |B_p(r)|\, d\mathrm{vol}_g(p)$.
Then, for every $0<r<\pi$,
$$
\overline V(r)
\le
\omega_{n-1}
\int_0^r
\sin^{n-1}t
\fint_N
\left[
1+\frac{1-t\cot t}{n}
\left(R_g(p)-n(n-1)\right)
\right]^{-\frac12}
\, d\mathrm{vol}_g(p)\, dt.
$$
\end{theorem}

Note that $1+\frac{1-t\cot t}{n}\left(R_g(p)-n(n-1)\right)>1$ for $0<t<\pi$ if $R_g(p)>n(n-1)$.

We now describe the main ingredients of the proof. The first is a coefficient-adapted Jacobian comparison, which is implicitly stated in \cite{BrownFreedman2022}. Let $z=(p, v)\in SN$, the unit tangent bundle. Let $\gamma_z$ be the geodesic with initial velocity $v$, and let $F_z(t)$ be its polar Jacobian. Consider the scalar Jacobi equation
$$
j_z''(t) + \frac{\operatorname{Ric}_{\gamma_z(t)} (\gamma_z'(t), \gamma_z'(t))}{n-1}\, j_z(t) =0,
\qquad
j_z(0)=0,
\qquad
j_z'(0)=1.
$$
Up to the first conjugate time along $\gamma_z$, one has
\begin{equation}\label{eq:intro-jacobian-comparison}
F_z(t)\le j_z(t)^{n-1}.
\end{equation}
This comparison does not require any curvature lower bound.

The second ingredient is an integral Jacobi comparison under a measure-preserving flow, inspired by the shuffling principle of Brown and Freedman \cite{BrownFreedman2022}. For a solution of an ordinary differential equation on $[0, \infty)$ with prescribed initial data, we define its stopped solution to agree with the original solution up to its first positive zero, if one exists, and to be identically zero thereafter. Starting from their finite-family comparison and passing to an integral formulation, we obtain the following version for measure-preserving flows.

\begin{theorem}[Theorem~\ref{thm:continuum-shuffle}]
\label{thm:intro-integral-jacobi}
Let $(X, m)$ be a compact metric space equipped with a Borel probability measure $m$, let $K$ be a continuous function on $X$, and let $\{\Phi_t\}_{0\le t\le T}$ be a continuous measure-preserving flow on $X$. For each $z\in X$, let $j_z$ be the stopped solution of
$$
j_z''(t)+K(\Phi_tz)j_z(t)=0,
\qquad
j_z(0)=0,
\qquad
j_z'(0)=1.
$$
Then, for every positive integer $\mu$,
$$
\int_X j_z(T)^\mu\, dm(z)
\le
\int_X\widehat{s}_{K(z)}(T)^\mu\, dm(z),
$$
where $\widehat{s}_a$ is the stopped solution of
$$
s''+as=0,
\qquad
s(0)=0,
\qquad
s'(0)=1.
$$
\end{theorem}

To apply the preceding result, we use a third ingredient: Liouville's theorem for the geodesic flow. Recall that, for $z=(p, v)\in SN$, the geodesic flow $G_t\colon SN\to SN$ is given by
$$
G_tz=\left(\gamma_z(t), \gamma_z'(t)\right).
$$
Let $m$ be the normalised Liouville measure on $SN$. Liouville's theorem \cite[Corollary~1.31 and Exercise~1.33]{Paternain} states that
$$
(G_t)_*m=m
$$
for every $t\in\mathbb R$. This invariance is also used in \cite{Kwong2025} to pass from estimates along individual geodesics to averaged estimates over the unit tangent bundle.

Define
\begin{equation}\label{eq:K}
K(p, v)=\frac{\operatorname{Ric}_p(v, v)}{n-1}.
\end{equation}
Applying Theorem~\ref{thm:intro-integral-jacobi} to $(SN, m)$ with $\Phi_t=G_t$ and coefficient function $K$, and then combining it with \eqref{eq:intro-jacobian-comparison}, reduces the volume estimate to the model scalar Jacobi solutions $\widehat s_{K(z)}$. Explicit integration of these model solutions gives the inequality in \eqref{eq:intro-determinant-bound}.

Jiang, Li, and Wang~\cite{JiangLiWang} observed that Theorem~\ref{thm:Bray conjecture} can be applied to obtain a volume comparison theorem under a pointwise lower bound on the Branson $Q$-curvature. Their argument applies the maximum principle to the defining equation for $Q_g$, thereby obtaining a pointwise scalar-curvature lower bound. Our scalar-curvature integral estimate yields a different formulation that requires only a weighted integral lower bound on $Q_g$, rather than a pointwise lower bound. This is done by using an integration-by-parts argument, which reduces the problem to Theorem~\ref{thm:intro-main2}; see Theorem~\ref{thm:integral Q curvature}.

The remainder of the paper is organized as follows. In Section~\ref{sec:some scalar}, we establish several scalar comparison formulas through the analysis of Riccati-type equations. In Section~\ref{sec: integral comp}, we extend the finite shuffling comparison for stopped Jacobi solutions to the integral setting of a measure-preserving flow. Finally, in Section~\ref{sec:geometric applications}, we apply this integral comparison to obtain estimates for the average areas and volumes of metric spheres and balls, the determinant and scalar-curvature volume estimates, and a volume comparison theorem under a weighted integral condition on the $Q$-curvature.

\begin{ack}
We would like to thank Michael Freedman for explaining the shuffling principle and Xu-Qian Fan and Miles Simon for helpful discussions.
\end{ack}

\section{ Some scalar comparisons using Riccati analysis}\label{sec:some scalar}

We begin by introducing the notation and establishing several comparison formulas through the analysis of Riccati-type equations. These formulas will be used later. Proposition~\ref{prop:fixed-model} is not needed for the proof of the main result, but may be of independent interest. No curvature bound is assumed in these comparisons. Standard facts about Jacobi tensors, cut points, and Riccati analysis may be found in \cite{BishopCrittenden, Petersen}.

For $k\in\mathbb R$, define
\begin{equation*}\label{not:sk}
s_k(t)=
\begin{cases}
\dfrac{\sin(\sqrt{k}\, t)}{\sqrt{k}}, & k>0, \\
t, & k=0, \\
\dfrac{\sinh(\sqrt{-k}\, t)}{\sqrt{-k}}, & k<0,
\end{cases}
\qquad
F_k(t)=s_k(t)^{n-1},
\end{equation*}
and
$$\operatorname{Ric}_k=\operatorname{Ric}-(n-1)kg. $$

If $z=(p, v)\in SN$, let $\gamma_z$ be the unit speed geodesic with initial vector $v$ starting from $p$. Let $J_z(t):v^\perp\to \gamma_z'(t)^\perp$ be the normal Jacobi tensor with $J_z(0)=0$ and $J_z'(0)=I$. Equivalently, for each $u \in v^{\perp}$, let $Y_u$ be the Jacobi field satisfying
$$
Y_u(0)=0, \quad \nabla_{v} Y_u(0)=u.
$$
The normal Jacobi tensor $J_z(t)$ is defined by
$$ J_z(t) u=Y_u(t). $$
Before the first conjugate point, set
$$
F_z(t)=\det J_z(t),
\qquad
A_z(t)=J_z'(t)J_z(t)^{-1},
\qquad
H_z(t)=\operatorname{tr}A_z(t).
$$
Thus $F_z$ is the polar Jacobian, $A_z(t)$ is the shape operator of the geodesic sphere of radius $t$ at the point $\exp_p(tv)$ when $t$ is small, and $H_z=(\log F_z)'$ is the mean curvature. The traceless part of $A_z$ is denoted by $\stackrel{\circ}{A}_z$.

Let $c(z)$ be the cut time of $\gamma_z$. We let
$$\widehat F_z(t)=F_z(t)\mathbf 1_{[0, c(z))}(t). $$
Let $d\theta_p$ denote the usual measure on $S_pN$, and let $\omega_{n-1}=\lvert\mathbb S^{n-1}\rvert$. We normalize the Liouville measure by
\begin{equation}\label{eq:normalized-liouville}
dm(p, v)=\frac{1}{\lvert N\rvert\omega_{n-1}}\, d\mathrm{vol}_g(p)d\theta_p(v).
\end{equation}
Thus $m$ is a probability measure, i.e. $m(SN)=1$.
\begin{lemma}\label{lem:polar-volume}
For a closed Riemannian manifold,
\begin{equation}\label{eq:polar-volume}
\lvert N\rvert
=\omega_{n-1}\int_{SN}\int_0^{c(z)}F_z(t)\, dt\, dm(z)
=\omega_{n-1}\int_0^\infty\int_{SN}\widehat F_z(t)\, dm(z)\, dt.
\end{equation}
\end{lemma}

\begin{proof}
For each fixed $p$, geodesic polar coordinates up to the cut locus give
$$
\lvert N\rvert
=\int_{S_pN}\int_0^{c(p, v)}F_{(p, v)}(t)\, dt\, d\theta_p(v).
$$
Integrate this identity over $p\in N$, divide by $\lvert N\rvert$, and use \eqref{eq:normalized-liouville}. This proves the first equality in \eqref{eq:polar-volume}. The second follows from Fubini-Tonelli's theorem and the definition of $\widehat F_z$.
\end{proof}
The following result generalizes \cite[Proposition 2.1]{Kwong2025}.
\begin{proposition}\label{prop:fixed-model}
Suppose that $s_k>0$ on $(0, r]$ and that $r$ is smaller than the first conjugate time along $\gamma_z$. Define
$$h_z(t)=H_z(t)-(n-1)\frac{s_k'(t)}{s_k(t)}. $$
Then
\begin{equation} \label{eq:exact-fixed}
\log\frac{F_z(r)}{F_k(r)}
=-\int_0^r\int_0^\tau \frac{s_k(t)^2}{s_k(\tau)^2} \left[\operatorname{Ric}_k(\gamma_z'(t), \gamma_z'(t)) +\lvert\stackrel\circ{A}_z(t)\rvert^2 +\frac{h_z(t)^2}{n-1} \right] \, dt\, d\tau.
\end{equation}
In particular,
\begin{equation}\label{eq:kwong-jacobian}
F_z(r)\le F_k(r) \exp\left[-\int_0^r\int_0^\tau \frac{s_k(t)^2}{s_k(\tau)^2} \operatorname{Ric}_k(\gamma_z'(t), \gamma_z'(t)) \, dt\, d\tau\right].
\end{equation}
\end{proposition}

\begin{proof}
Taking the trace of the Riccati equation $A_z'+A_z^2+\mathcal R_{\gamma_z'}=0$ gives $H_z'+\lvert A_z\rvert^2+\operatorname{Ric}(\gamma_z', \gamma_z')=0$, and $\lvert A_z\rvert^2=\frac{H_z^2}{ n-1 }+ \lvert\stackrel{\circ}{A}_z\rvert^2$, where $\mathcal R_{\gamma_z'}:=R(\cdot, \gamma_z')\gamma_z'$. Hence
\begin{equation}\label{eq:trace-riccati}
H_z'+\frac{H_z^2}{n-1} +\lvert\stackrel{\circ}{A}_z\rvert^2 +\operatorname{Ric}(\gamma_z', \gamma_z')=0.
\end{equation}
The function $H_k:=(n-1)\frac{s_k'}{s_k}$ satisfies
$$H_k'+\frac{H_k^2}{n-1}+(n-1)k=0. $$
Subtract this equation from \eqref{eq:trace-riccati}. Since $H_z=H_k+h_z$, we obtain
\begin{equation*}\label{eq:h-equation}
h_z'+2\frac{s_k'}{s_k}h_z+\frac{h_z^2}{n-1} +\lvert\stackrel{\circ}{A}_z\rvert^2 +\operatorname{Ric}_k(\gamma_z', \gamma_z')=0.
\end{equation*}
Multiplication by the integrating factor $s_k^2$ gives
$$
(s_k^2h_z)'=-s_k^2 \left(\operatorname{Ric}_k(\gamma_z', \gamma_z') +\lvert\stackrel{\circ}{A}_z\rvert^2 +\frac{h_z^2}{n-1} \right).
$$
The expansions $F_z(t)=t^{n-1}(1+O(t^2))$ and $s_k(t)=t+O(t^3)$ show that $s_k(t)^2h_z(t)\to0$ as $t\downarrow0$. Integrating first from $0$ to $\tau$, we have
\begin{align*}
s_k(\tau)^2h_z(\tau)^2=-\int_{0}^{\tau}s_k(t)^2\left(\operatorname{Ric}_k\left(\gamma_z^{\prime}(t), \gamma_z^{\prime}(t)\right)+\left|\stackrel{\circ}{A}_z(t)\right|^2+\frac{h_z(t)^2}{n-1}\right)dt.
\end{align*}
Divide both sides by $s_k(\rho)$, using
$$\left(\log\frac{F_z}{F_k}\right)'=h_z
$$
and integrating from $0$ to $r$, we obtain \eqref{eq:exact-fixed}. Dropping the two nonnegative terms gives \eqref{eq:kwong-jacobian}.
\end{proof}

\begin{definition}\label{def:scalar-jacobi}
For $z\in SN$, let $y_z$ solve
\begin{equation}\label{eq:scalar-jacobi}
y_z''(t)+\frac{1}{n-1} \operatorname{Ric}(\gamma_z'(t), \gamma_z'(t))y_z(t)=0, \qquad y_z(0)=0, \qquad y_z'(0)=1.
\end{equation}
Let $r(z)$ be its first positive zero. We define
$$
\widehat y_z(t)=
\begin{cases}
y_z(t), & 0\le t<r(z), \\
0, & t\ge r(z).
\end{cases}
$$
\end{definition}
Similar to Proposition \ref{prop:fixed-model}, we can also compare $F_z(r)$ with $y_z(r)^{n-1}$ as follows.
\begin{proposition}\label{prop:adapted}
Before the first conjugate point and the first zero of $y_z$, if
$$
d_z(t):= H_z(t)-(n-1)\frac{y_z'(t)}{y_z(t)},
$$
then
\begin{align}
\label{eq:exact-shear}
\log\frac{F_z(r)}{y_z(r)^{n-1}}
=& -\int_0^r\int_0^\tau \frac{y_z(t)^2}{y_z(\tau)^2} \left(\left|\stackrel{\circ}{A}_z(t)\right|^2 +\frac{d_z(t)^2}{n-1} \right) \, dt\, d\tau
\end{align}
In particular,
\begin{equation}\label{eq:adapted-first}
F_z(r)\le y_z(r)^{n-1}.
\end{equation}
The equality holds only if $\stackrel \circ A_z(t)=0$ for $0<t<r$.

If, in addition, $s_k>0$ on $(0, r]$, then
\begin{align} \label{eq:comparison-hierarchy}
F_z(r) \le y_z(r)^{n-1}
\le & F_k(r) \exp\left[-\int_0^r\int_0^\tau \frac{s_k(t)^2}{s_k(\tau)^2} \operatorname{Ric}_k \left(\gamma_z'(t), \gamma_z'(t)\right) \, dt\, d\tau \right].
\end{align}
\end{proposition}

\begin{proof}
Put
$$H_y=(n-1)\frac{y_z'}{y_z}. $$
Equation \eqref{eq:scalar-jacobi} gives
$$
H_y' +\frac{H_y^2}{n-1} +\operatorname{Ric}(\gamma_z', \gamma_z') =0.
$$
Subtracting this equation from \eqref{eq:trace-riccati}, and recalling that $d_z=H_z-H_y$, gives
$$
d_z' +2\frac{y_z'}{y_z}d_z +\frac{d_z^2}{n-1} +\left|\stackrel{\circ}{A}_z\right|^2
=0.
$$
Equivalently,
$$
\left(y_z^2d_z\right)' = -y_z^2 \left[\left|\stackrel{\circ}{A}_z\right|^2 +\frac{d_z^2}{n-1} \right].
$$
Since $y_z(t)=t+O(t^3)$ and $d_z(t)=O(t)$ as $t\downarrow0$,
integration gives
$$
d_z(\tau)
=-\int_0^\tau \frac{y_z(t)^2}{y_z(\tau)^2} \left[\left|\stackrel{\circ}{A}_z(t)\right|^2 +\frac{d_z(t)^2}{n-1} \right] \, dt.
$$
On the other hand,
$$
\frac{d}{d\tau} \log\frac{F_z(\tau)}{y_z(\tau)^{n-1}}
=H_z(\tau)-H_y(\tau)
=d_z(\tau).
$$
Using $\lim_{\tau\downarrow0} \frac{F_z(\tau)}{y_z(\tau)^{n-1}} =1$
and integrating once more gives \eqref{eq:exact-shear}.

For the comparison \eqref{eq:comparison-hierarchy}, let
$$
H_k=(n-1)\frac{s_k'}{s_k}, \qquad e_z=H_y-H_k.
$$
The equations satisfied by $H_y$ and $H_k$ give
$$
e_z' +2\frac{s_k'}{s_k}e_z +\frac{e_z^2}{n-1} +\operatorname{Ric}_k(\gamma_z', \gamma_z') =0,
$$
and hence
$$
\left(s_k^2e_z\right)' =-s_k^2 \left[\operatorname{Ric}_k(\gamma_z', \gamma_z') +\frac{e_z^2}{n-1} \right].
$$
Since $s_k(t)=t+O(t^3)$ and $e_z(t)=O(t)$ as $t\downarrow0$, two integrations yield
\begin{align*}
\log\frac{y_z(r)^{n-1}}{F_k(r)}
=& -\int_0^r\int_0^\tau \frac{s_k(t)^2}{s_k(\tau)^2} \left[\operatorname{Ric}_k \left(\gamma_z'(t), \gamma_z'(t)\right) +\frac{e_z(t)^2}{n-1} \right] \, dt\, d\tau.
\end{align*}
The two identities above give \eqref{eq:comparison-hierarchy}.
\end{proof}

\section{Integral comparison for stopped Jacobi solutions}\label{sec: integral comp}

We first define the notion of a stopped solution. Let $q\in L^\infty[0, T]$ and let $u_q$ be the solution of
\begin{equation}\label{eq:ordinary-scalar-jacobi}
u_q''+q(t)u_q=0, \qquad u_q(0)=0, \qquad u_q'(0)=1.
\end{equation}
Define
$$\tau_q=\inf\{t>0:u_q(t)=0\}, $$
where $\tau_q=\infty$ if $u_q$ has no positive zero. The corresponding stopped solution is defined as
\begin{equation}\label{eq:definition-stopped-solution}
j_q(t)=
\begin{cases}
u_q(t), & 0\le t\le \tau_q, \\
0, & t>\tau_q.
\end{cases}
\end{equation}
This convention is natural for volume comparison, where the radial volume density is nonnegative and each radial geodesic is integrated only up to its cut time.

Let $a\in \mathbb R$, we define $\widehat s_a$ to be the stopped solution of
$$
s''+aj=0, \qquad s(0)=0, \qquad s'(0)=1.
$$
So for $a>0$,
\begin{equation}\label{eq:stopped-sine}
\widehat{s}_a(t)=
\begin{cases}
\frac{\sin (\sqrt{a} t)}{\sqrt{a}}, & 0 \le t<\frac{\pi}{\sqrt{a}}, \\
0, & t \ge \frac{\pi}{\sqrt{a}}.
\end{cases}
\end{equation}
For $a\le0$, $\widehat s_a=s_a$.

We use the following finite-family comparison of Brown and Freedman. The proof can be found in \cite[Section~2.5] {BrownFreedman2022}, in the paragraph following Eq. (42).

\begin{proposition}\cite{BrownFreedman2022}
\label{prop:finite-BF}
Let $q_1, \ldots, q_L\in L^\infty[0, T]$. For almost every $t\in[0, T]$, arrange their values in nondecreasing order:
$$q_{(1)}(t)\le q_{(2)}(t)\le \cdots\le q_{(L)}(t), $$
where
$$
\{q_{(1)}(t), \ldots, q_{(L)}(t)\} = \{q_1(t), \ldots, q_L(t)\}
$$
as multisets (i.e. sets with multiplicities). Let $j_\ell$ and $j_{(\ell)}$ be the stopped solutions of
$$
j_\ell''+q_\ell(t)j_\ell=0, \qquad j_{(\ell)}''+q_{(\ell)}(t)j_{(\ell)}=0,
$$
with
$$
j_\ell(0)=j_{(\ell)}(0)=0, \qquad j_\ell'(0)=j_{(\ell)}'(0)=1.
$$
Then, for every positive integer $\mu$,
\begin{equation*}\label{eq:finite-BF-general}
\sum_{\ell=1}^{L}j_\ell(T)^\mu \le \sum_{\ell=1}^{L}j_{(\ell)}(T)^\mu.
\end{equation*}
In particular, suppose that there are real numbers $a_1, \ldots, a_L$ such that
$$
\{q_1(t), \ldots, q_L(t)\} = \{a_1, \ldots, a_L\}
$$
as multisets for almost every $t\in[0, T]$. Then
\begin{equation*}\label{eq:finite-BF}
\sum_{\ell=1}^{L}j_\ell(T)^\mu
\le
\sum_{\ell=1}^{L}\widehat s_{a_\ell}(T)^\mu.
\end{equation*}
\end{proposition}

We next prove two limiting statements needed to pass from the finite shuffling comparison to its integral form. Since our treatment of this passage differs somewhat from that of Brown and Freedman, we include the relevant details. The first statement concerns the dependence of a stopped solution on its coefficient.

The following lemma is a standard application of Gronwall's inequality.
\begin{lemma}
\label{lem:linear-system-estimate}
Let $A\in L^1([0, T], \mathbb R^{m\times m})$ and $b\in L^1([0, T], \mathbb R^m)$. Suppose that $z\colon[0, T]\to\mathbb R^m$ is absolutely continuous and satisfies $z(0)=0$ and $z'(t)=A(t)z(t)+b(t)$ for almost every $t\in[0, T]$. Then
\begin{equation*}\label{eq:linear-system-estimate}
\sup_{0\le t\le T}|z(t)|
\le \exp\left(\int_0^T|A(s)|\, ds\right) \left(\int_0^T|b(s)|\, ds \right).
\end{equation*}
Here $|A|$ is the operator norm.
\end{lemma}

\begin{lemma}
\label{lem:stopped-continuity}
Let $q_\nu, q\in L^\infty[0, T]$ and suppose that
$$
\sup_\nu\lVert q_\nu\rVert_{L^\infty[0, T]}<\infty, \qquad \lVert q_\nu-q\rVert_{L^1[0, T]}\longrightarrow0.
$$
Let $j_{q_\nu}$ and $j_q$ be the stopped solutions defined by \eqref{eq:ordinary-scalar-jacobi} and \eqref{eq:definition-stopped-solution}. Then
\begin{equation*}\label{eq:stopped-continuity}
j_{q_\nu}(T)\longrightarrow j_q(T).
\end{equation*}
\end{lemma}

\begin{proof}
Let $u_\nu$ and $u$ denote the ordinary, unstopped solutions associated with $q_\nu$ and $q$, respectively. Put
$$
Y_\nu=
\begin{pmatrix}
u_\nu\\
u_\nu'
\end{pmatrix},
\qquad
Y=
\begin{pmatrix}
u\\
u'
\end{pmatrix}.
$$
These vector-valued functions satisfy
$$
Y_\nu'=A_\nu Y_\nu,
\qquad
Y'=AY,
$$
where $
A_\nu(t)=
\begin{pmatrix}
0& 1\\
-q_\nu(t)& 0
\end{pmatrix}, \;
A(t)=
\begin{pmatrix}
0& 1\\
-q(t)& 0
\end{pmatrix}
$.

The uniform bound for the coefficients and Gronwall's inequality give a common uniform bound for $Y_\nu$ and $Y$ on $[0, T]$. Since
$$
(Y_\nu-Y)' = A_\nu(Y_\nu-Y)+(A_\nu-A)Y,
$$
applying Lemma~\ref{lem:linear-system-estimate} to $Y_\nu-Y$ gives
\begin{equation*}\label{eq:ordinary-C1-convergence}
\lVert u_\nu-u\rVert_{C^1[0, T]}
\le C\lVert q_\nu-q\rVert_{L^1[0, T]},
\end{equation*}
where $C$ depends only on $T$ and a common $L^\infty$ bound for the coefficients. In particular,
$$ u_\nu\longrightarrow u \qquad\text{in }C^1[0, T]. $$

Let $\tau$ be the first positive zero of $u$, with $\tau=\infty$ if no such zero exists. We distinguish three cases.

Suppose first that $\tau>T$. Since $u'(0)=1$, $C^1$ convergence of $u_\nu$ implies that there is a small $\delta>0$ such that $u_\nu'>0$ on $[0, \delta]$ for sufficiently large $\nu$, and so $u_\nu>0$ there for such $\nu$. On the other hand, $\min_{[\delta, T]}u>0$, so the uniform convergence implies that $u_\nu>0$ on $[\delta, T]$. Thus neither $u$ nor $u_\nu$ is stopped before $T$, and
$$j_{q_\nu}(T)=u_\nu(T)\longrightarrow u(T)=j_q(T). $$

Suppose next that $\tau<T$. We necessarily have $u'(\tau)<0$, for otherwise $u(\tau)=u'(\tau)=0$ implies $u\equiv0$ by uniqueness, contradicting $u'(0)=1$. We may therefore choose $\varepsilon>0$ such that
$$
\tau+\varepsilon<T, \qquad u(\tau-\varepsilon)>0, \qquad u(\tau+\varepsilon)<0.
$$
For all sufficiently large $\nu$, the same two strict inequalities hold with $u_\nu$ in place of $u$. Hence $u_\nu$ has a positive zero before $T$, and consequently
$$j_{q_\nu}(T)=j_q(T)=0. $$

Finally, suppose that $\tau=T$. If $u_\nu$ has a positive zero at or before $T$, then $j_{q_\nu}(T)=0$. Otherwise $j_{q_\nu}(T)=u_\nu(T)>0$. In both cases,
$$0\le j_{q_\nu}(T)\le |u_\nu(T)|. $$
Since $u_\nu(T)\to u(T)=0$, it follows that $j_{q_\nu}(T)\to0=j_q(T)$.
\end{proof}

The following lemma gives an integral version of the finite shuffling comparison. For the geometric application, one may regard $X$ as $SN$, so that each $z$ labels a geodesic and its associated approximating scalar Jacobi equation. On the $r$-th time interval, the coefficient of this equation is constant and equal to $q_r(z)$, assumed to be one of $a_1, \ldots, a_s$. This value approximates the Ricci coefficient along the corresponding geodesic on that time interval. Condition \eqref{eq:equal-level-measures} says that, for each $i$, the measure of the points assigned the value $a_i$ is the same on every time interval, although the set of such points may change from one interval to the next. The right-hand side of \eqref{eq:finite-valued-integral} corresponds to the configuration in which each $z$ retains its initial coefficient $q_1(z)$ throughout $[0, T]$. The lemma states that, starting from this configuration, rearranging the coefficients between time intervals, while preserving the measure assigned to each value $a_i$, cannot increase the averaged $\mu$-th power of the terminal values.

\begin{lemma}
\label{lem:finite-valued-integral}
Let $(X, m)$ be a probability measure space, and let $0=t_0<t_1<\cdots<t_{\mathrm{N}}=T$. Let $a_1, \ldots, a_s$ are distinct real numbers. Suppose we have a family of measurable functions
$$q_r:X\longrightarrow\{a_1, \ldots, a_s\}$$
for $r=1, \cdots, {\mathrm{N}}$, such that
\begin{equation}\label{eq:equal-level-measures}
m\{z:q_1(z)=a_i\}=
m\{z:q_2(z)=a_i\}=\cdots=
m\{z:q_{\mathrm{N}}(z)=a_i\} \text{ for each } a_i.
\end{equation}
For each $z\in X$, define
$$
\mathrm{p}_z(t)=q_r(z) \qquad\text{when }t_{r-1}<t<t_r,
$$
and let $j_z:=j_{\mathrm{p}_z}$ be its stopped solution. Then, for every $\mu\in \mathbb { N }$,
\begin{equation}\label{eq:finite-valued-integral}
\int_Xj_z(T)^\mu\, dm(z)
\le \int_X\widehat s_{q_1(z)}(T)^\mu\, dm(z).
\end{equation}
\end{lemma}
\begin{proof}
For $\alpha=(\alpha_1, \ldots, \alpha_{\mathrm{N}})$, where each $\alpha_i\in\{1, \cdots, s\}$, let
$$
E_\alpha
=
\{z\in X:
q_1(z)=a_{\alpha_1}, \ldots, q_{\mathrm{N}}(z)=a_{\alpha_{\mathrm{N}}}\},
\qquad
\lambda_\alpha=m(E_\alpha).
$$
The sets $\{E_\alpha\}_{\alpha\in \{1, \cdots, s\}^{\mathrm{N}}}$ form a partition of $X$.

Define the piecewise constant function $q_\alpha$ on $[0, T]$ by
$$
q_\alpha(t)=a_{\alpha_r}
\qquad\text{for }t_{r-1}<t<t_r.
$$
If $z\in E_\alpha$, then $\mathrm{p}_z=q_\alpha$. Let $J_\alpha$ denote the
value at time $T$ of the stopped solution of
$$
j''+q_\alpha(t)j=0,
\qquad
j(0)=0,
\qquad
j'(0)=1.
$$
Thus
$$
j_z(T)=J_\alpha
\qquad\text{for }z\in E_\alpha.
$$
The assumption \eqref{eq:equal-level-measures} gives
\begin{equation}\label{eq:balancing-identities}
\sum_{\alpha:\, \alpha_r=i}\lambda_\alpha = \sum_{\alpha:\, \alpha_1=i}\lambda_\alpha
\qquad
(1\le r\le {\mathrm{N}}, \ 1\le i\le s).
\end{equation}
Thus, for each $i$, the total weight assigned to the value $a_i$ is the
same on every time interval.

Choose nonnegative rational numbers $\lambda_\alpha^{(\nu)}$ such that $\lambda_\alpha^{(\nu)}\longrightarrow\lambda_\alpha$, \quad $\sum_\alpha\lambda_\alpha^{(\nu)}=1$, and such that \eqref{eq:balancing-identities} remains valid with
$\lambda_\alpha^{(\nu)}$ in place of $\lambda_\alpha$:
\begin{equation}\label{eq:balancing-identities'}
\sum_{\alpha: \alpha_r=i} \lambda^{(\nu)}_\alpha=\sum_{\alpha: \alpha_1=i} \lambda^{(\nu)}_\alpha \quad(1 \le r \le {\mathrm{N}}, 1 \le i \le s).
\end{equation}

For each $\nu$, choose a positive integer $L_\nu$ such that $L_\nu\lambda_\alpha^{(\nu)}$ is an integer for every $\alpha$. Form a family of $L_\nu$ ``coefficient functions'' by taking $L_\nu\lambda_\alpha^{(\nu)}$ copies of each $q_\alpha$. By \eqref{eq:balancing-identities'}, the coefficient values occurring on every time interval form the same multiset. After arranging these values in nondecreasing order at each time, the resulting family of functions are therefore constant functions (they are exactly $a_{1}, \cdots, a_s$ with multiplicities). Proposition \ref{prop:finite-BF} gives $\sum_\alpha L_\nu \lambda_\alpha^{(\nu)} J_\alpha^\mu \le \sum_\alpha L_\nu \lambda_\alpha^{(\nu)} \widehat{s}_{a_{\alpha_1}}(T)^\mu$, and so
$$
\sum_\alpha \lambda_\alpha^{(\nu)} J_\alpha^\mu \le \sum_\alpha \lambda_\alpha^{(\nu)} \widehat{s}_{a_{\alpha_1}}(T)^\mu.
$$
Letting $\nu\to\infty$, we obtain
$$
\sum_\alpha\lambda_\alpha J_\alpha^\mu
\le
\sum_\alpha\lambda_\alpha
\widehat s_{a_{\alpha_1}}(T)^\mu.
$$
Finally,
$$
\sum_\alpha\lambda_\alpha J_\alpha^\mu = \int_Xj_z(T)^\mu\, dm(z)
\text{ and }
\sum_\alpha\lambda_\alpha \widehat s_{a_{\alpha_1}}(T)^\mu = \int_X\widehat s_{q_1(z)}(T)^\mu\, dm(z).
$$
This proves \eqref{eq:finite-valued-integral}.
\end{proof}
We now extend the preceding comparison from coefficients taking values in a finite set to the continuously varying coefficients arising in the geometric application. Let $K$ be defined by \eqref{eq:K}, and let $G_t$ denote the geodesic flow. Each unit-speed geodesic is determined by its initial condition $z\in SN$, and we are going to consider the associated family of equations $j_z''(t)+K(G_tz)j_z(t)=0$, indexed by $z\in SN$. Although the coefficient $K(G_tz)$ varies with $t$ along each geodesic, preservation of Liouville measure implies that, for every Borel set $E\subset\mathbb R$,
$$
m\{z\in SN:K(G_tz)\in E\} = m\{z\in SN:K(z)\in E\}.
$$
Thus $K(G_t z)$ are rearranged among the geodesics as time varies, while the measure assigned to each range of values remains unchanged. This is the continuous counterpart of the condition in the preceding lemma. It allows the subsequent volume estimate to retain the directional Ricci curvature, rather than using only its common lower bound, and subsequently leads to an improvement over Bishop's volume bound.

Again, it is convenient to state this integral Jacobi comparison in the slightly more general setting of a measure-preserving flow. The theorem below shows that replacing the varying coefficients $K(\Phi_tz)$ by the constant coefficients $K(z)$ gives an upper bound after integration over the underlying space.
\begin{theorem} \label{thm:continuum-shuffle}
Let $X$ be a compact metric space equipped with a Borel probability measure $m$. Let $\Phi:[0, \infty)\times X\longrightarrow X$ be a continuous flow, and write $\Phi_t=\Phi(t, \cdot)$. Supppose $m(\Phi_t^{-1}(E))=m(E)$ for every Borel set $E\subset X$ and every $t\ge 0$. Let $K$ be a continuous function on $X$. For each $z\in X$, let $j_z$ be the stopped solution of
\begin{equation*}\label{eq:flow-jacobi}
j''(t)+K(\Phi_tz)j(t)=0, \qquad j(0)=0, \qquad j'(0)=1.
\end{equation*}
Then for every positive integer $\mu$ and every $t\ge 0$,
\begin{equation*}\label{eq:continuum-shuffle}
\int_Xj_z(t)^\mu\, dm(z) \le \int_X\widehat s_{K(z)}(t)^\mu\, dm(z).
\end{equation*}
\end{theorem}

\begin{proof}
Fix $T>0$. Since $K(X)$ is a compact subset of $\mathbb R$, there is a sequence of Borel functions $K_{\mathrm{N}}\colon X\to\mathbb R$, each taking values in a finite set, such that
\begin{equation}\label{eq:uniform-finite-valued-approximation}
\|K_{\mathrm{N}}-K\|_{L^\infty(X)}\longrightarrow0.
\end{equation}
In particular,
\begin{equation}\label{eq:uniform-coefficient-bound}
\sup_{\mathrm{N}}\|K_{\mathrm{N}}\|_{L^\infty(X)}+\|K\|_{L^\infty(X)}\le M
\end{equation}
for some $M<\infty$.

Let $t_r=\frac{rT}{{\mathrm{N}}}$, \;$0\le r\le {\mathrm{N}}$,
and define the piecewise constant coefficient
\begin{equation}\label{eq:discrete-flow-coefficient}
q_{{\mathrm{N}}, z}(t)=K_{\mathrm{N}}(\Phi_{t_{r-1}}z)
\qquad
\text{for }t_{r-1}<t<t_r.
\end{equation}
Let $j_{{\mathrm{N}}, z}$ be its stopped solution.

For every value $a$ taken by $K_{\mathrm{N}}$, measure preservation gives
\begin{align*}
m\{z:K_{\mathrm{N}}(\Phi_{t_{r-1}}z)=a\}
& =
m\left(\Phi_{t_{r-1}}^{-1}\{z:K_{\mathrm{N}}(z)=a\}\right)\\
& =
m\{z:K_{\mathrm{N}}(z)=a\}.
\end{align*}

Thus the coefficient functions $z\longmapsto K_{\mathrm{N}}(\Phi_{t_{r-1}}z)$ satisfy \eqref{eq:equal-level-measures}. Since $\Phi_0$ is the identity, Lemma \ref{lem:finite-valued-integral} gives
\begin{equation}\label{eq:approximating-integral-comparison}
\int_Xj_{{\mathrm{N}}, z}(T)^\mu\, dm(z)
\le \int_X\widehat s_{K_{\mathrm{N}}(z)}(T)^\mu\, dm(z).
\end{equation}

The map $(t, z)\longmapsto K(\Phi_tz)$ is uniformly continuous on $[0, T]\times X$. It follows
from \eqref{eq:uniform-finite-valued-approximation} and
\eqref{eq:discrete-flow-coefficient} that as $\mathrm{N}\to\infty$,
\begin{align*}
& \sup_{z\in X}\sup_{0\le t\le T} \left|q_{{\mathrm{N}}, z}(t)-K(\Phi_tz)\right|\\
\le& \lVert K_{\mathrm{N}}-K\rVert_{L^\infty(X)} + \sup_{\substack{z\in X, \;s, t\in[0, T] |s-t|\le T/{\mathrm{N}}}} \left|K(\Phi_sz)-K(\Phi_tz)\right| \longrightarrow0.
\end{align*}
Lemma \ref{lem:stopped-continuity} therefore implies that, for every $z\in X$, $\displaystyle \lim_{\mathrm{N}\to \infty}j_{{\mathrm{N}}, z}(T)= j_z(T)$.

Let $u_{{\mathrm{N}}, z}$ denote the ordinary, unstopped solution associated with $q_{{\mathrm{N}}, z}$. By \eqref{eq:uniform-coefficient-bound}, the corresponding first-order system has uniformly bounded coefficients. As in Lemma \ref{lem:stopped-continuity}, Gronwall's inequality gives a constant $C_T$, independent of ${\mathrm{N}}$ and $z$,
such that
$$
|u_{{\mathrm{N}}, z}(t)|+|u_{{\mathrm{N}}, z}'(t)|\le C_T
\qquad
(0\le t\le T).
$$
Since $j_{{\mathrm{N}}, z}$ agrees with $u_{{\mathrm{N}}, z}$ before its first positive zero and is zero thereafter,
$$
0\le j_{{\mathrm{N}}, z}(T)^\mu\le C_T^\mu.
$$
The same bound holds, after increasing $C_T$ if necessary, for $\widehat s_{K_{\mathrm{N}}(z)}(T)^\mu$. Moreover, as $\mathrm{N} \rightarrow \infty$,
$$
\widehat s_{K_{\mathrm{N}}(z)}(T) \longrightarrow \widehat s_{K(z)}(T),
$$
because $a\mapsto\widehat s_a(T)$ is continuous on $\mathbb R$. Applying dominated convergence theorem in \eqref{eq:approximating-integral-comparison} then gives
$$
\int_Xj_z(T)^\mu\, dm(z) \le \int_X\widehat s_{K(z)}(T)^\mu\, dm(z).
$$

Since $T>0$ was arbitrary, the theorem follows.
\end{proof}
\begin{lemma}\label{lem:flow-shuffle}
Let $G_t:SN\to SN$ be geodesic flow and define $K(p, v)=\frac{\operatorname{Ric}_p(v, v)}{n-1}$. Then $G_t$ preserves the normalized Liouville measure $m$ and
$$
\frac{1}{n-1}
\operatorname{Ric}_{\gamma_z(t)}(\gamma_z'(t), \gamma_z'(t))
=K(G_tz).
$$
\end{lemma}

\section{Geometric applications}\label{sec:geometric applications}
In this section, we apply the integral Jacobi comparison to derive the main geometric results of the paper. We first obtain comparison estimates for the average areas and volumes of metric spheres and balls. We then establish the averaged Ricci-spectrum and determinant bounds, from which the main volume comparison results, Theorems~\ref{thm:intro-main} and
\ref{thm:intro-main2}, follow.
\subsection{Average volumes of metric balls}
For $p\in N$ and $r\ge 0$, let
$$
B_p(r)=\{x\in N:d(p, x)\le r\},
\qquad
S_p(r)=\{x\in N:d(p, x)=r\}.
$$
These are the metric ball and metric sphere of radius $r$ centred at $p$. They need not be the images under $\exp_p$ of the corresponding ball and sphere in $T_pN$ when $r$ is greater than the injectivity radius. Let
$$
V(p, r)=|B_p(r)|,
\qquad
A(p, r)=\mathcal H^{n-1}(S_p(r)),
$$
and define their averages by
\begin{equation*}\label{eq:average-area-volume}
\overline A(r)=\fint_N A(p, r)\, d\mathrm{vol}_g(p),
\qquad
\overline V(r)=\fint_N V(p, r)\, d\mathrm{vol}_g(p).
\end{equation*}

\begin{lemma}\label{lem:average-metric-polar}
For every $r\ge 0$,
\begin{equation}\label{eq:average-metric-volume-polar}
\overline V(r)
=\omega_{n-1}\int_{SN}\int_0^r \widehat F_z(t)\, dt\, dm(z).
\end{equation}
Moreover, for almost every $r>0$,
\begin{equation}\label{eq:average-metric-area-polar}
\overline A(r)
=\omega_{n-1}\int_{SN}\widehat F_z(r)\, dm(z).
\end{equation}
In particular, $\overline V$ is absolutely continuous and
\begin{equation}\label{eq:average-volume-derivative}
\overline V'(r)=\overline A(r)
\end{equation}
for almost every $r>0$.
\end{lemma}

\begin{proof}
For each fixed $p\in N$, geodesic polar coordinates up to the cut locus give
$$
V(p, r) = \int_{S_pN}\int_0^r \widehat F_{(p, v)}(t)\, dt\, d\theta_p(v).
$$
Integrating over $p\in N$, dividing by $|N|$ and using \eqref{eq:normalized-liouville} proves \eqref{eq:average-metric-volume-polar}.

The distance function $d_p(x)=d(p, x)$ is Lipschitz and satisfies
$|\nabla d_p|=1$ almost everywhere. The coarea formula
therefore gives
$$V(p, r)=\int_0^r A(p, t)\, dt. $$
Averaging it over $p$ yields
$$\overline V(r)=\int_0^r \overline A(t)\, dt. $$
Comparison with \eqref{eq:average-metric-volume-polar} proves
\eqref{eq:average-metric-area-polar} and
\eqref{eq:average-volume-derivative} for almost every $r>0$.
\end{proof}
\begin{remark}
The condition ``for almost every $r$'' in \eqref{eq:average-metric-area-polar} is necessary. At exceptional radii, a metric sphere may contain part of the cut locus of positive $(n-1)$-dimensional measure. The volume comparison below nevertheless holds for every radius.
\end{remark}

Recall that $K(p, v)=\frac{\operatorname{Ric}_p(v, v)}{n-1}$. We define
\begin{align}
A_{\operatorname{Ric}}(r)
& =
\omega_{n-1}\int_{SN}
\widehat s_{K(z)}(r)^{n-1}\, dm(z),
\label{eq:directional-Ricci-area}\\
V_{\operatorname{Ric}}(r)
& =
\omega_{n-1}\int_{SN}\int_0^r
\widehat s_{K(z)}(t)^{n-1}\, dt\, dm(z).
\label{eq:directional-Ricci-volume}
\end{align}

\begin{theorem}\label{thm:average-metric-comparison}
Let $(N^n, g)$ be a closed Riemannian manifold.
Then
\begin{equation}\label{eq:average-metric-area-comparison}
\overline A(r)\le A_{\operatorname{Ric}}(r)
\end{equation}
for almost every $r>0$, and
\begin{equation}\label{eq:average-metric-volume-comparison}
\overline V(r)\le V_{\operatorname{Ric}}(r)
\end{equation}
for every $r\ge 0$.
\end{theorem}

\begin{proof}
Let $y_z$ be the solution of \eqref{eq:scalar-jacobi}, and let
$\widehat y_z$ be its stopped version. Proposition~\ref{prop:adapted} gives
$$F_z(t)\le y_z(t)^{n-1}$$
before the first conjugate point and the first positive zero of $y_z$.
This inequality also shows that the first positive zero of $y_z$ cannot occur before the cut time.
It follows that
\begin{equation}\label{ineq:F y}
\widehat F_z(t)\le \widehat y_z(t)^{n-1}
\end{equation}
for every $z\in SN$ and every $t\ge 0$.

By Lemma~\ref{lem:flow-shuffle}, the coefficient in
\eqref{eq:scalar-jacobi} is $K(G_tz)$ and the geodesic flow preserves the normalized Liouville measure. Theorem~\ref{thm:continuum-shuffle}, applied
with $X=(SN, m)$, $\Phi_t=G_t$ and $\mu=n-1$, then gives
\begin{equation}\label{eq:average-stopped-y-comparison}
\int_{SN}\widehat y_z(t)^{n-1}\, dm(z)
\le
\int_{SN}\widehat s_{K(z)}(t)^{n-1}\, dm(z)
\end{equation}
for every $t\ge 0$.

For almost every $r>0$, Lemma~\ref{lem:average-metric-polar},
\eqref{ineq:F y}, and
\eqref{eq:average-stopped-y-comparison} give
\begin{align*}
\overline A(r)
& =
\omega_{n-1}\int_{SN}\widehat F_z(r)\, dm(z)\\
& \le
\omega_{n-1}\int_{SN}\widehat y_z(r)^{n-1}\, dm(z)\\
& \le
\omega_{n-1}\int_{SN}
\widehat s_{K(z)}(r)^{n-1}\, dm(z)
=
A_{\operatorname{Ric}}(r).
\end{align*}
This proves \eqref{eq:average-metric-area-comparison}. Integrating the same
inequalities in $t$ over $[0, r]$ and using the Fubini--Tonelli theorem gives
\begin{align*}
\overline V(r)
& =
\omega_{n-1}\int_0^r\int_{SN}
\widehat F_z(t)\, dm(z)\, dt\\
& \le
\omega_{n-1}\int_0^r\int_{SN}
\widehat y_z(t)^{n-1}\, dm(z)\, dt\\
& \le
\omega_{n-1}\int_{SN}\int_0^r
\widehat s_{K(z)}(t)^{n-1}\, dt\, dm(z),
\end{align*}
which is \eqref{eq:average-metric-volume-comparison}.
\end{proof}

When the Ricci curvature is positive, the right-hand sides of
\eqref{eq:average-metric-area-comparison} and
\eqref{eq:average-metric-volume-comparison} can be written more explicitly.

\begin{corollary}\label{cor:positive-Ricci-metric-comparison}
Under the assumptions of Theorem~\ref{thm:average-metric-comparison},
suppose in addition that $\operatorname{Ric}>0$. Define
$$
\mathcal I_n(\rho)=\int_0^\rho\sin^{n-1}u\, du,
\qquad 0\le \rho\le \pi.
$$
Then, for almost every $r>0$,
\begin{align}
\overline A(r)
\le {}&
\omega_{n-1}\int_{SN}
K(z)^{-\frac{n-1}{2}}
\sin^{n-1}\left(\sqrt{K(z)}\, r\right)
\mathbf 1_{\{\sqrt{K(z)}\, r<\pi\}}
\, dm(z),
\label{eq:positive-Ricci-average-area}
\end{align}
and, for every $r\ge 0$,
\begin{align}
\overline V(r)
\le {}&
\omega_{n-1}\int_{SN}
K(z)^{-\frac n2}
\mathcal I_n\left(\min\{\sqrt{K(z)}\, r, \pi\}\right)
\, dm(z).
\label{eq:positive-Ricci-average-volume}
\end{align}
\end{corollary}

\begin{proof}
For $a>0$, the stopped constant-coefficient solution $\widehat s_a$ is given by \eqref{eq:stopped-sine}.
Substitution into \eqref{eq:directional-Ricci-area} proves
\eqref{eq:positive-Ricci-average-area}. Moreover, the change of variable
$u=\sqrt a\, t$ gives
$$
\int_0^r\widehat s_a(t)^{n-1}\, dt
=
a^{-\frac n2}
\mathcal I_n\left(\min\{\sqrt a\, r, \pi\}\right).
$$
Substitution into \eqref{eq:directional-Ricci-volume} proves
\eqref{eq:positive-Ricci-average-volume}.
\end{proof}

Whenever $s_k>0$ on $(0, r]$, we define
$$
A_k(r)=\omega_{n-1}F_k(r)
=\omega_{n-1}s_k(r)^{n-1},
\qquad
V_k(r)=\int_0^r A_k(t)\, dt.
$$

For $r>0$ such that $s_k>0$ on $(0, r]$, define the positive functions
$$
\tau_k(r)
=
\begin{cases}
\dfrac{1-\sqrt{k}\, r\cot(\sqrt{k}\, r)}{2k},
& k>0, \\
\dfrac{r^2}{6},
& k=0, \\
\dfrac{1-\sqrt{-k}\, r\coth(\sqrt{-k}\, r)}{2k},
& k<0.
\end{cases}
$$
Indeed, these functions are obtained from the equation
$\tau_k(r)=\int_0^r \int_0^\tau \frac{s_k(t)^2}{s_k(\tau)^2} d t d \tau$, and are introduced in \cite{Kwong2025}.
For $x\ge 0$, define
\begin{equation*}\label{eq:definition-F-n}
\mathcal F_n(x)
=
\frac{
\displaystyle
\int_{-1}^1
e^{-xu^2}(1-u^2)^{\frac{n-3}{2}}\, du
}{
\displaystyle
\int_{-1}^1
(1-u^2)^{\frac{n-3}{2}}\, du
}\le1.
\end{equation*}
Thus $\mathcal F_n(0)=1$, and $\mathcal F_n$ is strictly decreasing on
$[0, \infty)$.

\begin{lemma}\label{lem:exponential-Ricci-trace}
Let $L$ be a nonnegative self-adjoint endomorphism of $\mathbb R^n$. Then
\begin{equation}\label{eq:exponential-Ricci-trace}
\frac{1}{\omega_{n-1}}
\int_{\mathbb S^{n-1}}
e^{- \langle L\theta, \theta\rangle}\, d\theta
\le
\mathcal F_n\left(\, \operatorname{tr}L\right).
\end{equation}
Moreover,
\begin{equation}\label{eq:F-n-elementary-bound}
\mathcal F_n(x)
\le
\left(1+\frac{2x}{n}\right)^{-\frac12}.
\end{equation}
\end{lemma}

\begin{proof}
Let $\lambda_1, \ldots, \lambda_n\ge 0$ be the eigenvalues of $L$, and
$S=\sum_{i=1}^n\lambda_i=\operatorname{tr}L$.
For $\lambda=(\lambda_1, \ldots, \lambda_n)$, define
$$
\Phi(\lambda)
=
\frac{1}{\omega_{n-1}}
\int_{\mathbb S^{n-1}}
\exp\left(- \sum_{i=1}^n\lambda_i\theta_i^2\right)\, d\theta.
$$
The function $\Phi$ is convex because the exponential of an affine function
of $\lambda$ is convex, and integration preserves convexity. It is also
symmetric in $\lambda_1, \ldots, \lambda_n$.

If $S=0$, the result is immediate. Suppose that $S>0$.
We can write
$$
(\lambda_1, \ldots, \lambda_n)
=
\sum_{i=1}^n\frac{\lambda_i}{S}
(0, \ldots, 0, \overbrace{S}^{\textrm{$i$-th}}, 0, \ldots, 0),
$$
By convexity,
$$
\Phi(\lambda_1, \ldots, \lambda_n)
\le
\sum_{i=1}^n\frac{\lambda_i}{S}\Phi(S, 0, \ldots, 0)
=
\Phi(S, 0, \ldots, 0).
$$
The standard integration formulas (cf. \cite[Equation~(1.16)]{AtkinsonHan}) and $\omega_{n-1}=\omega_{n-2}\int_{-1}^1
(1-u^2)^{\frac{n-3}{2}}\, du$ give
$$
\Phi(S, 0, \ldots, 0)
=
\frac{
\displaystyle
\int_{-1}^1
e^{- Su^2}(1-u^2)^{\frac{n-3}{2}}\, du
}{
\displaystyle
\int_{-1}^1
(1-u^2)^{\frac{n-3}{2}}\, du
}
=
\mathcal F_n(S),
$$
which proves \eqref{eq:exponential-Ricci-trace}.

To prove \eqref{eq:F-n-elementary-bound}, observe that since $e^y\ge 1+y$,
$$
e^{-xu^2}=
\left(e^{-\frac{2x}{n }u^2}\right)^{-\frac{n}{2}}
\le \left(1+\frac{2x}{n}u^2\right)^{-\frac n2}.
$$
Therefore
\begin{align*}
\mathcal F_n(x)
\le \frac{ \displaystyle \int_{-1}^1 \left(1+\frac{2x}{n}u^2\right)^{-\frac n2} (1-u^2)^{\frac{n-3}{2}}\, du }{ \displaystyle \int_{-1}^1 (1-u^2)^{\frac{n-3}{2}}\, du }
=\left(1+\frac{2x}{n}\right)^{-\frac12}.
\end{align*}
The last equality follows from the integration formula \cite[Equation~(1.16)]{AtkinsonHan} and
Lemma~\ref{lem:spherical-det}, applied to $L=\operatorname{diag}\left(1+\frac{2x}{n}, 1, \ldots, 1\right)$, for which $\langle L\theta, \theta\rangle=1+\frac{2x}{n}\theta_1^2$ and $\det L=1+\frac{2x}{n}$.
\end{proof}

\begin{proposition}\label{prop:K-fixed-model}
Let $k\in \mathbb R$. Suppose $\mathrm{Ric}\ge (n-1)k$ on $N$ and $s_k>0$ on $(0, r]$. Then for any $z\in SN$,
\begin{equation}\label{eq:K-fixed-model}
\widehat s_{K(z)}(r)^{n-1}
\le
F_k(r)
\exp\left(
-(n-1)\tau_k(r)(K(z)-k)
\right).
\end{equation}
\end{proposition}

\begin{proof}
Fix $z\in SN$. Before the first positive zero of $s_{K(z)}$, let
$e_z = (n-1)\left(\frac{s_{K(z)}'}{s_{K(z)}} - \frac{s_k'}{s_k} \right)$.
The equations satisfied by $s_{K(z)}$ and $s_k$ give
$$
e_z' + 2\frac{s_k'}{s_k}e_z + \frac{e_z^2}{n-1} + (n-1)(K(z)-k) = 0.
$$
Proceeding as in the proof of Proposition~\ref{prop:adapted}, we obtain
\begin{align*}
\log\frac{s_{K(z)}(r)^{n-1}}{F_k(r)}
={}&
-\int_0^r\int_0^\tau
\frac{s_k(t)^2}{s_k(\tau)^2}
\left(
(n-1)(K(z)-k)+\frac{e_z(t)^2}{n-1}
\right)\, dt\, d\tau\\
\le {}&
-(n-1)\tau_k(r)(K(z)-k).
\end{align*}
Exponentiating proves \eqref{eq:K-fixed-model} before the first positive
zero of $s_{K(z)}$. After that zero, $\widehat s_{K(z)}(r)=0$ by
definition, so the inequality continues to hold.
\end{proof}

We now prove an averaged volume comparison theorem that incorporates the scalar-curvature excess in addition to the Ricci curvature lower bound. Under $\operatorname{Ric}\ge (n-1)kg$, the quantity $R_k=R-n(n-1)k$ is nonnegative, and the estimate below quantitatively refines the Bishop--Gromov bound for the average volumes of balls whenever this excess is not identically zero.
\begin{theorem}\label{thm:scalar-excess-metric-balls}
Let $(N^n, g)$ be a closed Riemannian manifold satisfying $\operatorname{Ric}\ge (n-1)kg$,
and define $R_k=R-n(n-1)k$.
Let $r>0$ be such that $s_k>0$ on $(0, r]$. Then, for almost every such $r$,
\begin{align}
\overline A(r)
& \le
A_k(r)\fint_N
\mathcal F_n\left(\tau_k(r)R_k(p)\right)\, d\mathrm{vol}_g(p)
\label{eq:scalar-excess-area}\\
& \le
A_k(r)\fint_N
\left(
1+\frac{2\tau_k(r)}{n}R_k(p)
\right)^{-\frac12}
\, d\mathrm{vol}_g(p).
\label{eq:elementary-scalar-excess-area}
\end{align}
For every such $r$,
\begin{align}
\overline V(r)
& \le
\int_0^r
A_k(t)\fint_N
\mathcal F_n\left(\tau_k(t)R_k(p)\right)
\, d\mathrm{vol}_g(p)\, dt
\label{eq:scalar-excess-volume}\\
& \le
\int_0^r
A_k(t)\fint_N
\left(
1+\frac{2\tau_k(t)}{n}R_k(p)
\right)^{-\frac12}
\, d\mathrm{vol}_g(p)\, dt.
\label{eq:elementary-scalar-excess-volume}
\end{align}
\end{theorem}

\begin{proof}
Theorem \ref{thm:average-metric-comparison} gives
\begin{equation}\label{eq:area-before-scalar-excess}
\overline A(r)
\le
\omega_{n-1}\int_{SN}
\widehat s_{K(z)}(r)^{n-1}\, dm(z)
\end{equation}
for almost every $r$.

Proposition \ref{prop:K-fixed-model} gives
\begin{equation}\label{eq:constant-exponential-comparison}
\widehat s_{K(z)}(r)^{n-1}
\le
F_k(r)
\exp\left(
-(n-1)\tau_k(r)(K(z)-k)
\right).
\end{equation}

At each $p\in N$, define the nonnegative self-adjoint endomorphism $L_p = \frac{\operatorname{Ric}_p^{\sharp}}{n-1}-kI$.
Then $\operatorname{tr} L_p=\frac{R_k(p)}{n-1}$.
Applying Lemma~\ref{lem:exponential-Ricci-trace} with $(n-1)\tau_k(r)L_p$ gives
\begin{align*}
& \frac{1}{\omega_{n-1}}
\int_{S_pN}
\exp\left(
-(n-1)\tau_k(r)(K(p, v)-k)
\right)\, d\theta_p(v)\\
\le&
\mathcal F_n\left(
(n-1)\tau_k(r)\operatorname{tr}L_p
\right)
=
\mathcal F_n\left(\tau_k(r)R_k(p)\right).
\end{align*}
Combining this inequality with
\eqref{eq:area-before-scalar-excess},
\eqref{eq:constant-exponential-comparison}, and
$A_k(r)=\omega_{n-1}F_k(r)$ gives
$$
\overline A(r)
\le
A_k(r)\fint_N
\mathcal F_n\left(\tau_k(r)R_k(p)\right)\, d\mathrm{vol}_g(p).
$$
This proves \eqref{eq:scalar-excess-area}.
Equation \eqref{eq:elementary-scalar-excess-area} follows from
\eqref{eq:F-n-elementary-bound}.

Since
$$\overline V(r)=\int_0^r \overline A(t)\, dt, $$
integration of the two area estimates proves
\eqref{eq:scalar-excess-volume} and
\eqref{eq:elementary-scalar-excess-volume}.
\end{proof}
\subsection{Total volume estimate}

The total-volume estimate is obtained by allowing the radius in
\eqref{eq:positive-Ricci-average-volume} to tend to infinity.

\begin{proposition}\label{prop:spectrum}
Let $(N^n, g)$ be a closed connected Riemannian manifold and suppose that
$\operatorname{Ric}>0$. Then
\begin{equation}\label{eq:spectrum-bound}
\frac{|N|}{|\mathbb S^n|}
\le
\int_{SN}
\left(
\frac{\operatorname{Ric}_p(v, v)}{n-1}
\right)^{-\frac n2}
\, dm(p, v).
\end{equation}
\end{proposition}

\begin{proof}
Since $N$ is connected and compact, $B_p(r)=N$ whenever
$r\ge \operatorname{diam}(N)$. Hence $\overline V(r)=|N|$ for all sufficiently large
$r$. Letting $r\to\infty$ in
\eqref{eq:positive-Ricci-average-volume} gives
\begin{align*}
|N|
& \le
\omega_{n-1}\int_{SN}
K(z)^{-\frac n2}
\left(\int_0^\pi\sin^{n-1}u\, du\right)dm(z)\\
& = |\mathbb S^n| \int_{S N}\left(\frac{\operatorname{Ric}_p(v, v)}{n-1}\right)^{-\frac{n}{2}} d m(p, v).
\end{align*}
\end{proof}

A direct estimate of $\int_{S_pN}\left(\frac{\operatorname{Ric}_p(v, v)}{n-1}\right)^{-\frac{n}{2}} \, d\theta$ in terms of the scalar curvature is not immediate. Indeed, since $x\mapsto x^{-\frac{n}{2}}$ is convex on $(0, \infty)$, Jensen's inequality gives a lower bound rather than an upper bound. However, the exponent $-\frac{n}{2}$ in \eqref{eq:spectrum-bound} is special: it allows the spherical average at each point to be evaluated exactly as the inverse square root of a determinant.

\begin{lemma}\label{lem:spherical-det}
Let $L$ be a positive definite self-adjoint automorphism of $\mathbb R^n$. Then
\begin{equation}\label{eq:spherical-det}
\frac{1}{\omega_{n-1}}\int_{\mathbb S^{n-1}}
\langle L\theta, \theta\rangle^{-\frac{n}{2}}\, d\theta
=(\det L)^{-\frac{1}{2}}.
\end{equation}
\end{lemma}

\begin{proof}
Using $\int_{\mathbb R}e^{-\lambda x^2}\, dx =\sqrt{\frac{\pi}{\lambda}}$, $\lambda>0$,
together with an orthogonal diagonalization of $L$ and Fubini's theorem, we obtain
\begin{equation}\label{eq:gaussian}
\int_{\mathbb R^n}e^{-\langle Lx, x\rangle}\, dx
=\pi^{\frac{n}{2}}(\det L)^{-\frac{1}{2}}.
\end{equation}
Recall that $\Gamma(\alpha)=\int_0^\infty e^{-u}u^{\alpha-1}\, du$ for $\alpha>0$.
For each fixed $\theta\in\mathbb S^{n-1}$, set
$a=\langle L\theta, \theta\rangle$ and use the substitution $u=ar^2$, we obtain
$$
\int_0^\infty
e^{-r^2\langle L\theta, \theta\rangle}r^{n-1}\, dr
=
\frac12\Gamma\left(\frac n2\right)
\langle L\theta, \theta\rangle^{-\frac n2}.
$$
Integrating over $\mathbb S^{n-1}$ and using polar coordinates yields
\begin{equation}\label{eq:spher-int}
\int_{\mathbb R^n}e^{-\langle Lx, x\rangle}\, dx
=
\frac12\Gamma\left(\frac n2\right)
\int_{\mathbb S^{n-1}}
\langle L\theta, \theta\rangle^{-\frac n2}\, d\theta.
\end{equation}
For $L=I$, the same calculation gives
$\pi^{\frac{n}{2}}=\frac12\Gamma(\frac{n}{2})\omega_{n-1}$.
Comparing \eqref{eq:spher-int} with \eqref{eq:gaussian} proves
\eqref{eq:spherical-det}.
\end{proof}
Recall that for a symmetric $(0, 2)$-tensor field $\sum_{i, j}T_{ij}dx^idx^j$ on $(N, g)$, we define its determinant (at a point) to be
\begin{align*}
\det{}_g T:= \det (T^{\#})=\det (T^i_j),
\end{align*}
where $T^i_j:=\sum_k g^{ik}T_{kj}$.

\begin{theorem}\label{thm:determinant}
Let $(N^n, g)$ be closed with
$\operatorname{Ric}>0$. Then
\begin{equation*}\label{eq:determinant-bound}
\frac{\lvert N\rvert}{\lvert\mathbb S^n\rvert}
\le \frac{1}{\lvert N\rvert} \int_N \det{}_g\left(\frac{\operatorname{Ric}}{n-1}\right)^{-\frac{1}{2}} \, d\mathrm{vol}_g.
\end{equation*}
The equality holds if and only if $(N, g)$ is isometric to a round sphere of positive constant curvature.
\end{theorem}

\begin{proof}
At $p\in N$, let $L_p=\frac{1}{n-1}\operatorname{Ric}_p^{\sharp}$.
Then
$\frac{\operatorname{Ric}_p(v, v)}{n-1}=\langle L_pv, v\rangle$. Split the integral
in \eqref{eq:spectrum-bound} first over $p$ and then over $S_pN$. By
Lemma \ref{lem:spherical-det},
\begin{align*}
\int_{SN} \left(\frac{\operatorname{Ric}_p(v, v)}{n-1}\right)^{-\frac{n}{2}}dm(p, v)
& =\frac{1}{\lvert N\rvert}
\int_N\left[\frac{1}{\omega_{n-1}}\int_{S_pN} \langle L_pv, v\rangle^{-\frac{n}{2}}\, d\theta_p(v) \right]d\mathrm{vol}_g(p)\\
& =\frac{1}{\lvert N\rvert} \int_N(\det L_p)^{-\frac{1}{2}}\, d\mathrm{vol}_g(p).
\end{align*}
Substitution into Proposition \ref{prop:spectrum} proves the theorem.

Suppose the equality holds, then \eqref{ineq:F y} is an equality. By Proposition \ref{prop:adapted}, $\stackrel{\circ}{A}_z(t)=0$ for all $z$ and for small enough $t$. Thus the sufficiently small geodesic spheres are umbilic. Using (cf. \cite[Theorem 3.1]{ChenVanhecke})
$$
A_z(t)=\frac{1}{t} I-\frac{t}{3} \mathcal {R}_{{\gamma}'_z(0)}+O\left(t^2\right),
$$
one obtains
$$
\mathcal {R}_v=\frac{\operatorname{Ric}(v, v)}{n-1} I \quad \text { on } v^{\perp}
$$
for every unit vector $v$. For orthonormal vectors $v, w \in T_p N$, this implies the sectional curvature
$K(v, w)=\left\langle \mathcal R_v w, w\right\rangle=\frac{\operatorname{Ric}(v, v)}{n-1}$.
By the symmetry of sectional curvature, $\operatorname{Ric}(v, v)=\operatorname{Ric}(w, w)$.
In general, if two unit vectors $v, w$ are not orthogonal, choose ($n \ge 3$) a unit vector $u$ orthogonal to both. It follows that
$$
\operatorname{Ric}(v, v)=\operatorname{Ric}(u, u)=\operatorname{Ric}(w, w).
$$
Hence, at each point, $\mathrm{Ric}=\lambda g$. So every sectional curvature at that point equals $\frac{\lambda}{n-1}$. Schur's lemma then shows that $\lambda$ is constant on $N$. Let $\frac{\lambda}{n-1}=\frac{1}{r^2}$. Then the equality implies $|N|_g=|\mathbb S^n|r^n$. But then $N$ is isometric to $\mathbb{S}^n(r) / \Gamma$, where $\Gamma$ is a finite group acting freely by isometries. The equality case says that $\frac{\left|\mathbb{S}^n\right|r^n}{|\Gamma|}=|N|_g=\left|\mathbb{S}^n\right|r^n$, which is possible only when $\Gamma$ is trivial and $N$ is isometric to a round sphere of positive curvature.
\end{proof}
The following algebraic inequality is elementary.
\begin{lemma}\label{lem:det-lower}
Let $a_1, \ldots, a_n\ge 0$. Then $$ \prod_{i=1}^n(1+a_i) \ge 1+\sum_{i=1}^n a_i. $$
Equality holds if and only if at most one of the numbers $a_i$ is positive.
\end{lemma}

\begin{theorem}\label{thm:main}
Let $(N^n, g)$ be a closed Riemannian manifold, where $n\ge 3$, and suppose
that $\operatorname{Ric}_g\ge (n-1)g$.
Then
\begin{equation}\label{eq:scalar-curvature-volume-bound}
\frac{|N|_g}{|\mathbb S^n|}
\le \frac{1}{|N|_g} \int_N \left(\frac{R_g}{n-1}-(n-1) \right)^{-\frac12} \, d\mathrm{vol}_g.
\end{equation}
The equality holds if and only if $N$ is isometric to the unit sphere.
\end{theorem}

\begin{proof}
At each $p\in N$, consider the self-adjoint automorphism $L_p=\frac{1}{n-1}\operatorname{Ric}(p)^\sharp$. The Ricci curvature lower bound implies that the eigenvalues of $L_p$ can be written as $1+a_i$, where $a_i\ge 0$. In particular, $\sum_{i=1}^n(1+a_i) = \operatorname{tr}L_p = \frac{R_g(p)}{n-1}$. Lemma~\ref{lem:det-lower} then gives
\begin{align*}
\det L_p
& =
\prod_{i=1}^n(1+a_i)\\
& \ge
1+\sum_{i=1}^n a_i
=
\frac{R_g(p)}{n-1}-(n-1).
\end{align*}
Therefore,
$$
\det\nolimits_g
\left(\frac{\operatorname{Ric}_g}{n-1}\right)^{-\frac12}
\le
\left(
\frac{R_g}{n-1}-(n-1)
\right)^{-\frac12}.
$$
The inequality now follows from Theorem~\ref{thm:determinant}.

Suppose the equality holds, then Theorem \ref{thm:determinant} shows that $N$ is isometric to a round sphere. On the other hand, by Lemma \ref{lem:det-lower}, $\operatorname{Ric}_g-(n-1) g$ is at most rank one at every point. Therefore $N$ has constant sectional curvature $1$ and so is the unit sphere.
\end{proof}

\begin{corollary}\label{cor:uniform-scalar-bound}
Let $n\ge 3$, let $\varepsilon\ge0$, and let $(N^n, g)$ be a closed
Riemannian manifold satisfying
\begin{equation*}\label{eq:main-assumptions}
\operatorname{Ric}_g\ge (n-1)g,
\qquad
R_g\ge n(n-1)(1+\varepsilon).
\end{equation*}
Then
\begin{equation*}\label{eq:main-bound}
|N|_g \le \frac{1}{\sqrt{1+n\varepsilon}} |\mathbb S^n|.
\end{equation*}
The equality holds if and only if $N$ is isometric to the unit sphere and $\varepsilon=0$.
\end{corollary}

\begin{proof}
The scalar curvature lower bound gives
$$
\frac{R_g}{n-1}-(n-1)
\ge n(1+\varepsilon)-(n-1)
=1+n\varepsilon.
$$
Substitution into \eqref{eq:scalar-curvature-volume-bound} proves the
corollary.
\end{proof}

Let us estimate the remaining gap from Bray's bound.
For $n\ge 2$ and $\varepsilon>0$, we have
\begin{equation*}\label{eq:strict-gap}
(1+n\varepsilon)^{-\frac{1}{2}}>(1+\varepsilon)^{-\frac{n}{2}}.
\end{equation*}
Indeed, the Bernoulli's inequality gives $(1+\varepsilon)^n>1+n\varepsilon$, which gives the above inequality.

The binomial expansions are
\begin{align*}
(1+n\varepsilon)^{-\frac{1}{2}}
=1-\frac n2\varepsilon+\frac{3n^2}{8}\varepsilon^2 +O(\varepsilon^3), \quad
(1+\varepsilon)^{-\frac{n}{2}}
=1-\frac n2\varepsilon+\frac{n(n+2)}{8}\varepsilon^2 +O(\varepsilon^3),
\end{align*}
Therefore, as $\varepsilon \downarrow 0$,
$$
(1+n \varepsilon)^{-\frac{1}{2}}-(1+\varepsilon)^{-\frac{n}{2}}=\frac{n(n-1)}{4} \varepsilon^2+O\left(\varepsilon^3\right).
$$
Thus the volume factor in our estimate agrees with Bray's predicted factor to first order in $\varepsilon$, but falls short of Bray's predicted bound by a positive term of order $\varepsilon^2$.

\subsection{Volume comparison involving the Q-curvature}
In this subsection, we prove a volume comparison result under an integral lower bound of the Branson $Q$-curvature.

Recall that for a Riemannian manifold $(N^n, g)$ of dimension $n\ge 3$, the Branson $Q$-curvature \cite{Branson1985} is defined by
\begin{equation}\label{eq:Q-curvature}
Q_g=-\frac{1}{2(n-1)}\Delta_gR_g-\frac{2}{(n-2)^2}|\stackrel {\circ}{\operatorname{Ric}_g}|_g^2+\frac{n^2-4}{8n(n-1)^2}R_g^2,
\end{equation}
where $\stackrel {\circ}{\operatorname{Ric}_g}$ denotes the traceless Ricci tensor. In particular, if $\operatorname{Ric}_g=(n-1)g$, then
$Q_g=\frac{n(n^2-4)}8$.

Volume comparison under assumptions on the $Q$-curvature was studied by Lin and Yuan~\cite{LinYuan2}. In particular, they proved that if $(N^n, \bar g)$ is a closed strictly stable Einstein manifold satisfying $\operatorname{Ric}_{\bar g}=(n-1)\bar g$, then every metric $g$ sufficiently close to $\bar g$ in the $C^4$-topology and satisfying
$$Q_g\ge \frac{n(n^2-4)}{8}$$
has volume at most $|N|_{\bar g}$.

Recently, Jiang, Li, and Wang obtained a volume comparison result \cite[Theorem 5.1]{JiangLiWang} involving $Q_g$ as an application of Theorem \ref{thm:Bray conjecture}. Their argument uses a pointwise lower bound for $Q_g$ and applies the maximum principle to \eqref{eq:Q-curvature} at a minimum point of $R_g$, thereby converting the $Q$-curvature assumption into a pointwise lower bound for the scalar curvature.

Our scalar-curvature integral estimate yields an alternative formulation involving a weighted integral lower bound for $Q_g$ rather than a pointwise lower bound:
\begin{theorem}\label{thm:integral Q curvature}
Let $(N^n, g)$ be a closed Riemannian manifold of dimension $n\ge 3$ satisfying
$\operatorname{Ric}_g\ge (n-1)g$.
Let $\varepsilon\ge 0$, and suppose that
\begin{equation}\label{eq:integral ass}
\int_N\frac{Q_g-\frac{n(n^2-4)}8(1+\varepsilon)^2}{\left(R_g-(n-1)^2\right)^2}\, d\mathrm{vol}_g\ge 0.
\end{equation}
Then
$|N|_g \le \frac{1}{\sqrt{1+n\varepsilon}}|\mathbb S^n|$.
Equality holds if and only if $\varepsilon=0$ and $(N, g)$ is isometric to the unit sphere.
\end{theorem}

\begin{proof}
Let $\widehat R:=R_g-(n-1)^2$. Since $\operatorname{Ric}_g\ge (n-1)g$, we have $\widehat R\ge n-1>0$. Let
$q_{n, \varepsilon}:=\frac{n(n^2-4)}{8}(1+\varepsilon)^2$ and
$c_n:=\frac{n^2-4}{8n(n-1)^2}$.

Multiplying
$$Q_g=-\frac{1}{2(n-1)}\Delta_gR_g-\frac{2}{(n-2)^2}|\stackrel {\circ}{\operatorname{Ric}_g}|^2+c_nR_g^2$$
by $\widehat R^{-2}$ and integrating by parts gives
\begin{align*}
\int_NQ_g\widehat R^{-2}\, d\mathrm{vol}_g
=& -\frac{1}{n-1}\int_N\widehat R^{-3}|\nabla R_g|^2\, d\mathrm{vol}_g-\frac{2}{(n-2)^2}\int_N|\stackrel {\circ}{\operatorname{Ric}_g}|^2\widehat R^{-2}\, d\mathrm{vol}_g+c_n\int_NR_g^2\widehat R^{-2}\, d\mathrm{vol}_g\\
\le& c_n\int_NR_g^2\widehat R^{-2}\, d\mathrm{vol}_g.
\end{align*}
Combining this with the assumption \eqref{eq:integral ass}, and using $R_g=\widehat R+(n-1)^2$, gives
$$\left(\frac{q_{n, \varepsilon}}{c_n}-(n-1)^4\right)\fint_N\widehat R^{-2}\, d\mathrm{vol}_g\le 1+2(n-1)^2\fint_N\widehat R^{-1}\, d\mathrm{vol}_g. $$

Let $y:=\left(\fint_N\widehat R^{-2}\, d\mathrm{vol}_g\right)^{\frac{1}{2}}$. By Cauchy--Schwarz inequality, $\fint_N\widehat R^{-1}\, d\mathrm{vol}_g\le y$. Since $\frac{q_{n, \varepsilon}}{c_n}=n^2(n-1)^2(1+\varepsilon)^2$, we obtain
$$\left(n^2(n-1)^2(1+\varepsilon)^2-(n-1)^4\right)y^2\le 1+2(n-1)^2y. $$
Equivalently,
$$\left(s^2-a^2\right) y^2-2 a y-1\le0, $$
where $s:=n(n-1)(1+\varepsilon)$ and $a:=(n-1)^2$. The LHS factorizes as $((s-a) y-1)((s+a) y+1)$, and so
$$\left(\fint_N \widehat{R}^{-2} d \operatorname{vol}_g\right)^{\frac{1}{2}}=y\le \frac{1}{s-a}=\frac{1}{(n-1)(1+n\varepsilon)}. $$
Hölder's inequality gives
$$\fint_N\widehat R^{-\frac12}\, d\mathrm{vol}_g\le \left(\fint_N\widehat R^{-2}\, d\mathrm{vol}_g\right)^{\frac14}\le \frac{1}{\sqrt{(n-1)(1+n\varepsilon)}}. $$
Since $\frac{R_g}{n-1}-(n-1)=\frac{\widehat R}{n-1}$, it follows that
$$\fint_N\left(\frac{R_g}{n-1}-(n-1)\right)^{-\frac12}\, d\mathrm{vol}_g\le \frac{1}{\sqrt{1+n\varepsilon}}. $$
Theorem \ref{thm:main} then implies
$$\frac{|N|_g}{|\mathbb S^n|}
\le \fint_N\left(\frac{R_g}{n-1}-(n-1)\right)^{-\frac{1}{2}} d \operatorname{vol}_g
\le \frac{1}{\sqrt{1+n\varepsilon}}. $$
The equality statement follows from the rigidity statement of Theorem \ref{thm:main}.
\end{proof}

\end{document}